\documentclass[12pt]{amsart}
\usepackage{amssymb}

\newtheorem{thm}{Theorem}[section]

\theoremstyle{remark}

\theoremstyle{definition}

\RequirePackage{amsfonts}
\RequirePackage{bbm}

\begin{document}

\title
[Direct Proof of the Reflection Principle]
{A Direct Proof of the Reflection Principle for Brownian Motion}

\author{S. J. Dilworth}
\address{Department of Mathematics, University of South Carolina,
		Columbia, SC 29208 U.S.A.}
\email{dilworth@math.sc.edu}
\thanks{The first author was supported by The National Science Foundation under Grant Number DMS-1361461.}

\author{Duncan Wright}
\address{Department of Mathematics, University of South Carolina,
		Columbia, SC 29208 U.S.A.}
\email{dw7@math.sc.edu}

\subjclass[2000]{Primary: 60J65 ; Secondary: 60G40}
\keywords{Brownian Motion; Reflection Principle; Stopping Times}

\begin{abstract}
We present a self-contained proof of the reflection principle for Brownian Motion.
\end{abstract}

\maketitle

\section{Introduction}
\noindent
The reflection principle proved below is  one of the most  important properties of Brownian Motion. So much so that any treatment of Brownian Motion would be incomplete without mentioning it and some of its many applications
(see e.g.\  \cite{S}). Most notable among these applications, using the hitting time $\tau_x=\inf\{t:B_t=x\}$, is that
$$P(\tau_x\le t)= 2P(B_t\ge x),$$
which in turn yields that $X_t:=\max_{0\le s\le t}B_s$ and $|B_t|$ have the same distribution. This famous result
is attributed to Louis Bachelier \cite[p. 197]{B}, and also, in a later more rigorous treatment, to Paul L\'evy \cite[p. 293]{L}.  In fact it was  Bachelier who first  introduced  the stochastic process, which  later on became known as Brownian Motion,  as a model for stock  prices in his pioneering work in mathematical finance.  Remarkably, \cite{B} precedes the rigorous construction  of Brownian Motion by almost two decades.

The reflection principle is invariably presented as a consequence of the Strong Markov Property. This approach has pedagogical value as it provides  one of the first applications of the Strong Markov Property (see e.g. \cite{SP}). However, it has the drawback of being beyond the scope of less specialized texts and consequently the proof of the reflection principle is often omitted. We present here a short and direct proof requiring few prerequisites which is intended to make the reflection principle more accessible.

Recall that a \textit{Standard Brownian Motion} (SBM) on a probability space $(\Omega,\mathcal{F},P)$ is a gaussian process $(B_t)_{t \ge 0}$ (i.e., the finite-dimensional distributions are  multivariate normal distributions), 
with $B_0=0$, continuous sample paths, $\mathbb{E} [B_t]=0$, and covariance function $\mathbb{E}[B_s B_t] = \min(s,t)$. The $\sigma$-algebra $\mathcal{F}_t$ is the smallest $\sigma$-algebra containing all $P$-null sets for which each $B_s$ ($0 \le s \le t$) is measurable. 

A stopping time   with respect to the \textit{standard Brownian filtration} $(\mathcal{F}_t)_{t\ge0}$  is a mapping  $T \colon \Omega \rightarrow  [0,\infty]$ satisfying $\{T \le t\} \in \mathcal{F}_t$ for each $t \ge 0$. $T$ is allowed to take the value $\infty$ with positive probability.

A tool that is used in our proof of the reflection principle is  the `uniqueness theorem': the fact  that the distribution of an $\mathbb{R}^n$-valued  random vector $X$ is determined by its characteristic function $\phi_{X}(\lambda) := \mathbb{E}[\exp(i \mathbf{\lambda}\cdot X)]$ ($\lambda \in \mathbb{R}^n$) (see e.g.\ \cite[p.\ 135]{M}). The uniqueness theorem is  used in a similar way  to prove  the Strong Markov Property in \cite{SP}. 

Our proof also uses standard properties of the conditional expectation operator with respect to a sub-$\sigma$-algebra $\mathcal{G}$, namely linearity and the fact that 
$\mathbb{E}[XY|\mathcal{G}] = X\mathbb{E}[Y|\mathcal{G}]$ for random variables $X,Y$ when $X$ is $\mathcal{G}$-measurable (see e.g.\
\cite[p.\ 187]{M}).
The
`independence of Brownian increments' is used in the following intuitively obvious but slightly tricky to prove  form:   if $n \ge 1$ and $s < t_1<\dots <t_n$ and $f \colon \mathbb{R}^n \rightarrow \mathbb{R}$ is bounded and continuous, then, setting $V:=f(B_{t_1}-B_s,\dots, B_{t_n}-B_s)$,
\begin{equation}  \label{eq: fact}\mathbb{E}[V|\mathcal{F}_s] = \mathbb{E}[V]. \end{equation}
For completeness a short proof of this standard fact is given at the end.

\section{Reflection Principle}

\begin{thm} (Reflection Principle) Let $(B_t)_{t\ge 0}$ be an SBM and let $T$ be a stopping time with respect to  $(\mathcal F_t)_{t\ge 0}$. Define
	$$
	B_t^T:=\begin{cases} B_t, & 0\le t \le T \\ 2B_T-B_t, & t>T. \end{cases}
	$$
	Then $(B_t^T)_{t\ge 0}$ is an SBM.
\end{thm}

\begin{proof}
 Note that $(B^T_t)$ clearly has continuous sample paths. By the uniqueness theorem, to complete the proof it is enough to show, for each $n\ge 1$ and $0< t_1<\cdots <t_n<\infty$, and constants $\lambda_j\in\mathbb{R}$ ($1\le j\le n$), that 
 $
 \mathbb E[e^{iX^T}]= \mathbb E[e^{iX}],
 $
 where $$X:=\sum_{j=1}^n \lambda_j B_{t_j}\quad \text{and}\quad X^T:=\sum_{j=1}^n \lambda_j B_{t_j}^T.$$  For notational convenience, set  $t_0 :=0$ and $t_{n+1} :=\infty$. First, suppose $T$ takes only finitely many values $0\le a_1<\cdots <a_m <\infty $. For each $1\le r\le m$, choose $k_r$ such that $t_{k_r}   \le a_r < t_{k_r+1}$ and let
 $$
 Y_r:=\sum_{j=1}^{k_r}\lambda_j B_{t_j} + \left( \sum_{j=k_r+1}^n \lambda_j\right) B_{a_r}\quad\text{and}\quad
 Z_r:=\sum_{j=k_r+1}^n \lambda_j (B_{t_j}-B_{a_r}).
 $$
 Note that $Y_r$ is $\mathcal F_{a_r}$-measurable,  $Z_r$ is independent of $\mathcal F_{a_r}$ by  \eqref{eq: fact}, and also that 
 $$
 X= \sum_{r=1}^m (Y_r+Z_r)\mathbbm 1_{\{T=a_r\}}\quad\text{and}\quad
 X^T = \sum_{r=1}^m (Y_r-Z_r)\mathbbm 1_{\{T=a_r\}}.
 $$
 Therefore
 \begin{align*}
 \mathbb E[e^{iX^T}]&= \sum_{r=1}^m \mathbb E[e^{i(Y_r-Z_r)}\mathbbm 1_{\{T=a_r\}}]\\
 &= \sum_{r=1}^m \mathbb E[e^{iY_r}\mathbbm 1_{\{T=a_r\}}\mathbb E[e^{-iZ_r} |\mathcal F_{a_r}]]
 \intertext{(since $e^{iY_r}\mathbbm 1_{\{T=a_r\}}$ is $\mathcal F_{a_r}$-measurable)}
 &= \sum_{r=1}^m \mathbb E[e^{iY_r}\mathbbm 1_{\{T=a_r\}}]\mathbb E[e^{-iZ_r}]
 \intertext{(by independence of $Z_r$ with respect to $\mathcal F_{a_r}$)}
 &= \sum_{r=1}^m \mathbb E[e^{iY_r}\mathbbm 1_{\{T=a_r\}}]\mathbb E[e^{iZ_r}]
 \intertext{(by symmetry of $Z_r$)}
 &=\sum_{r=1}^m \mathbb E[e^{i(Y_r+Z_r)}\mathbbm 1_{\{T=a_r\}}]
 = \mathbb E[e^{iX}]
 \end{align*} (by reversing the steps to get the first equality above).
 To extend  the result to a general stopping time $T$, we simply approximate  $T$ by stopping times $T_j$ which take only finitely many values. To make this precise, let $T_j(\omega) := 2^j$ if $T(\omega) > 2^j$ and
 $
 T_j(\omega):= k2^{-j}
 $
 if $(k-1)2^{-j}< T(\omega)\le k2^{-j}\le 2^j$. Then clearly $T_j\rightarrow T$ almost surely and, by continuity of the sample paths of $(B_t)$, $X^{T_j}\rightarrow X^T$ almost surely. Thus, by the bounded convergence theorem, 
 $$
 \mathbb E[e^{iX^T}]=\lim_{j\rightarrow \infty}\mathbb E[e^{iX^{T_j}}]=\mathbb E[e^{iX}].
 $$\end{proof} Finally, we prove \eqref{eq: fact}. By definition of the conditional expectation operator, we have to show that, for all $A\in\mathcal F_s$,\begin{equation} \label{eq:prove} \mathbb{E}[V \mathbbm{1}_A]  = \mathbb{E}[V] P(A).\end{equation}The collection $\mathcal{G}$ of all $A \in \mathcal{F}_s$ for which \eqref{eq:prove} holds is easily seen to be a \textit{monotone class} (i.e., $\mathcal{G}$ is closed under countable increasing unions and decreasing intersections) containing the $P$-null sets. Moreover, given $m \ge 1$ and $0< s_1 <\dots< s_m \le s$, $\mathcal{G}$ contains the $\sigma$-algebra $\sigma(B_{s_1},\dots, B_{s_m})$, the smallest $\sigma$-algebra for which each $B_{s_j}$ ($1 \le j \le m$) is measurable: this follows from independence of Brownian increments. The union over all of these $\sigma$-algebras as $m$ and $(s_j)_{j=1}^m$ vary is an algebra whose augmentation by the $P$-null sets  generates $\mathcal{F}_s$. The monotone class lemma  (see e.g.\ \cite[p.\ 4]{M}) now gives $\mathcal{G} = \mathcal{F}_s$.

\end{document}